\documentclass[11pt]{article} 
\usepackage{amssymb, latexsym}
\newtheorem{defin}{}
\newtheorem{saetze}[defin]{}
\newtheorem{conjec}[defin]{}
\newtheorem{lemmas}[defin]{}
\newtheorem{folger}[defin]{}
\newtheorem{bemerk}[defin]{}

\newenvironment{theorem}  {\begin{saetze}\it {\bf Theorem:}}{\end{saetze}}

\newenvironment{lemma}    {\begin{lemmas}\it {\bf Lemma:}}{\end{lemmas}}

\newenvironment{proof}    {{\it Proof}:}{{\hfill \fillbox \bigskip}}

\newcommand{\fillbox}{\mbox{$\bullet$}}
\newcommand{\ra}{\rightarrow}
\newcommand{\Ra}{\Rightarrow}
\newcommand{\La}{\Leftarrow}
\newcommand{\ms}{\mapsto}

\newcommand{\N}{\mathbb N}

\newcommand{\Z}{\mathbb Z}

\newcommand{\T}{\mathcal T}

\newenvironment{items}{\begin{list}{$\alph{item})$}
{\labelwidth18pt \leftmargin18pt \topsep3pt \itemsep1pt \parsep0pt}}
{\end{list}}

\begin{document}

\title{Hall polynomials for the torsion free nilpotent groups \\
       of Hirsch length at most 5}
\author{Bettina Eick and Ann-Kristin Engel}
\date{\today}
\maketitle

\begin{abstract}
A theorem by Hall asserts that the multiplication in torsion free 
nilpotent groups of finite Hirsch length can be facilitated by 
polynomials. In this note we exhibit explicit Hall polynomials for 
the torsion free nilpotent groups of Hirsch length at most 5.
\end{abstract}

\section{Introduction}

Let $\T$ denote the class of torsion free finitely generated nilpotent
groups. By \cite[p. 137]{Rob82}, each $\T$-group has a central series
with infinite cyclic quotients. The length of such a central series is
the Hirsch length of the group; see \cite[p. 152]{Rob82} for background.

Let $G$ be a group. We say that a sequence of elements $(g_1, \ldots, g_n)$
is a {\em $\T$-basis} for $G$ if the subgroups $G_i = \langle g_i, \ldots, 
g_n \rangle$ for $1 \leq i \leq n+1$ define a central series $G = G_1 > G_2 
> \ldots > G_n > G_{n+1} = \{1\}$ with infinite quotients $G_i/G_{i+1}$ for 
$1 \leq i \leq n$. The construction 
implies that each quotient $G_i/G_{i+1}$ is cyclic and generated by 
$g_i G_{i+1}$. It is not difficult to observe that a group $G$ has a 
$\T$-basis of length $n$ if and only if $G$ is $\T$-group of Hirsch length 
$n$.

We note that a $\T$-basis $(g_1, \ldots, g_n)$ of a $\T$-group $G$ is a 
polycyclic sequence for $G$ in the sense of \cite[Sec. 8.1]{HEO05}. Thus 
as in \cite[Lemma 8.3]{HEO05}, it follows that each element $g$ of $G$ can 
be written uniquely as $g = g_1^{a_1} \cdots g_n^{a_n}$ for certain $a_1, 
\ldots, a_n \in \Z$. Hence the $\T$-basis induces a bijection
\[ \beta : \Z^n \ra G : (a_1, \ldots, a_n) \ra g_1^{a_1} \cdots g_n^{a_n}. \]
The bijection $\beta$ is not a group homomorphism in general. More precisely,
$\beta$ translates the multiplication in $G$ to certain functions $p_i :  
\Z^n \times \Z^n \ra \Z$ for $1 \leq i \leq n$ with
\[ (g_1^{a_1} \cdots g_n^{a_n}) \cdot (g_1^{b_1} \cdots g_n^{b_n})
   = g_1^{p_1(a,b)} \cdots g_n^{p_n(a,b)}. \]
A fundamental theorem by Hall \cite{Hal69, Hal88} asserts that each function 
$p_i(a,b)$ can be described by a polynomial. These polynomials are nowadays 
called {\em Hall polynomials}.

Our aim here is to determine Hall polynomials for {\em all} $\T$-groups
of Hirsch length at most $5$.
Methods to compute Hall polynomials for a fixed $\T$-group given by a 
consistent polycyclic presentation have been exhibited by Sims 
\cite[page 441ff]{Sim94} and by Leedham-Green \& Soicher \cite{LGS98}. 
Our method is based on the approach by Sims and uses this in a generic
form so that it applies to the infinitely many $\T$-groups of Hirsch
length at most $5$ simultaneously.

An application of the Hall polynomials determined here is the classification 
of the torsion free nilpotent groups of Hirsch length at most $5$ as 
described in \cite{EEi16}.

\section{Nilpotent presentations and consistency}
\label{presen}

As a first step towards our aim we exhibit consistent nilpotent 
presentations for the $\T$-groups of Hirsch length at most $5$. We refer 
to \cite[Sec. 8.1]{HEO05} for details on this type of presentation.
For $n \in \N$ and $t = (t_{ijk} \mid 1 \leq i < j < k \leq n) \in 
\Z^{n \choose 3}$ we denote
\[ G(t) = \langle g_1, \ldots, g_n \mid 
           [g_j, g_i] = g_{j+1}^{t_{i,j,j+1}} \cdots g_n^{t_{i,j,n}}
            \mbox{ for } 1 \leq i < j \leq n \rangle.\]
Then the presentation defining $G(t)$ is a polycyclic presentation
and the relation imply that $G(t)$ is a nilpotent group of Hirsch
length at most $n$. This presentation is consistent if and only if 
$G(t)$ has Hirsch length precisely $n$, or, equivalently, if the 
generators of the presentation form a $\T$-basis for $G(t)$.

\begin{lemma} \label{cons}
Let $G$ be a $\T$-group of Hirsch length $5$ and let $(h_1, \ldots, h_5)$ 
be a $\T$-basis for $G$. Then there exist a unique $t \in \Z^{10}$ so that
the map $G(t) \ra G : g_i \ms h_i$ extends to an isomorphism. This $t$
satisfies the relations $t_{123} t_{345} = 0$ and $t_{124} t_{345} 
+ t_{145} t_{234} = t_{134} t_{245}$.
\end{lemma}

\begin{proof}
We first construct a suitable $t \in \Z^{10}$. The definition of $\T$-basis 
implies directly that $[h_j, h_i] \in G_{j+1}$ for each $1 \leq i < j 
\leq 5$. As each element of $G_{j+1}$ can be written (uniquely) as 
$h_{j+1}^{a_{j+1}} \cdots h_5^{a_5}$, this yields that there exist (unique)
$t_{ijk}$ with 
\[ [h_j, h_i] = h_{j+1}^{t_{i,j,j+1}} \cdots h_5^{t_{i,j,5}}\]
for $1 \leq i < j \leq 5$. Using this constructed $t \in \Z^{10}$ then
induces that the map $G(t) \ra G : g_i \ms h_i$ is an epimorphism. It
is an isomorphism, as $G$ has Hirsch length $5$ and $G(t)$ has at most
Hirsch length $5$. We now observe that the restrictions on $t$ are 
satisfied. For this purpose we evaluate the equation $(h_4 h_2) h_1 
= h_4 (h_2 h_1)$ in $G$ and obtain that
\begin{eqnarray*}
(h_4 h_2) h_1 
   &=& (h_2 h_4 h_5^{t_{245}}) h_1 \\ 
   &=& h_1 h_2^{h_1} h_4^{h_1} h_5^{t_{245}} \\
   &=& h_1 h_2 h_3^{t_{123}} h_4^{t_{124}} h_5^{t_{125}} 
       h_4 h_5^{t_{145}} h_5^{t_{245}} \\
   &=& h_1 h_2 h_3^{t_{123}} h_4^{t_{124}+1} h_5^{t_{125}+t_{145}+t_{245}}
       \mbox{ and } \\
h_4 (h_2 h_1)
   &=& h_4 (h_1 h_2 h_3^{t_{123}} h_4^{t_{124}} h_5^{t_{125}}) \\
   &=& h_1 h_4 h_2 h_3^{t_{123}} h_4^{t_{124}} h_5^{t_{125}+t_{145}}\\
   &=& h_1 h_2 h_4 h_3^{t_{123}} h_4^{t_{124}} 
       h_5^{t_{125}+t_{145}+t_{245}}\\
   &=& h_1 h_2 h_3^{t_{123}} h_4^{t_{124}+1} 
       h_5^{t_{125}+t_{145}+t_{245} + t_{123} t_{345}}.
\end{eqnarray*}
Comparison of the two results yields that $t_{123} t_{345} = 0$. 
Evaluating the equation $(h_3 h_2) h_1 = h_3 (h_2 h_1)$ in a similar form 
yields that $t_{124} t_{345} + t_{145} t_{234} = t_{134} t_{245}$.
\end{proof}

As a next step, we show that the restrictions on $t$ in Lemma \ref{cons}
yield that the presentation defining $G(t)$ is consistent.

\begin{lemma} \label{pres}
Let $n \in \{1, \ldots, 5\}$ and $t \in \Z^{n \choose 3}$.
\begin{items}
\item[\rm (a)]
Let $n \leq 4$. Then the presentation defining $G(t)$ is consistent.
\item[\rm (b)]
Let $n = 5$. 
Then the presentation defining $G(t)$ is consistent if and only if 
$t_{123} t_{345} = 0$ and $t_{124} t_{345} + t_{145} t_{234} 
= t_{134} t_{245}$. 
\end{items}
\end{lemma}

\begin{proof}
We prove (b) only and note that (a) follows by a similar (and easier)
calculation. \\
$\Ra:$ Suppose that $G(t)$ is consistent. Then the generators of this 
presentation form a $\T$-basis for $G(t)$. Thus Lemma \ref{cons} yields
the result. \\
$\La:$ Suppose that $t \in \Z^{10}$ with $t_{123} t_{345} = 0$ and 
$t_{124} t_{345} + t_{145} t_{234} = t_{134} t_{245}$ is given. We show
that $G(t)$ is consistent in this case. For this purpose we evaluate
the finitely many consistency relations, see 
\cite[page 424]{Sim94} (or \cite[Lemma 2.10]{Eic01}). 
In the case considered here, these consistency 
relations have the form
\[ (g_k g_j) g_i = g_k (g_j g_i) \mbox{ for } 1 \leq i < j < k \leq 5, 
  \mbox{ and }\]
\[ g_k = (g_k g_i^{-1})g_i \mbox{ for } 1 \leq i < k \leq 5.\]
It is not difficult to observe that these relations always hold if $k = 5$, 
since $g_5$ is central in $G$. Hence it remains to consider the relations
\begin{eqnarray*}
(R1) && (g_4 g_3) g_1 = g_4 (g_3 g_1), \\
(R2) && (g_4 g_2) g_1 = g_4 (g_2 g_1), \\
(R3) && (g_3 g_2) g_1 = g_3 (g_2 g_1), \\
(R4) && (g_4 g_3) g_2 = g_4 (g_3 g_2), \\
(R5) && g_2 = (g_2 g_1^{-1}) g_1, \\
(R6) && g_3 = (g_3 g_1^{-1}) g_1, \\
(R7) && g_4 = (g_4 g_1^{-1}) g_1, \\
(R8) && g_3 = (g_3 g_2^{-1}) g_2, \\
(R9) && g_4 = (g_4 g_2^{-1}) g_2, \\
(R10) && g_4 = (g_4 g_3^{-1}) g_3. 
\end{eqnarray*}
The relations $(R1)$ - $(R4)$ can be evaluated similar to the proof of
Lemma \ref{cons} and are satisfied due to the equations imposed on the 
exponents $t_{ijk}$. For the relations $(R5)$ - $(R10)$ we note that
\begin{eqnarray*}
g_4 g_1^{-1} &=& g_1^{-1} g_4 g_5^{-t_{145}}, \\
g_3 g_1^{-1} &=& g_1^{-1} g_3 g_4^{-t_{134}} g_5^{t_{134} t_{145}-t_{135}}, \\
g_2 g_1^{-1} &=& g_1^{-1} g_2 g_3^{-t_{123}} g_4^{t_{123} t_{134} - t_{124}}
                 g_5^{-t_{123} t_{134} t_{145} + t_{123} t_{135} 
                                  + t_{124} t_{145} - t_{125}}, \\
g_4 g_2^{-1} &=& g_2^{-1} g_4 g_5^{-t_{245}}, \\
g_3 g_2^{-1} &=& g_2^{-1} g_3 g_4^{-t_{234}} g_5^{t_{234} t_{245}-t_{235}}, \\
g_3 g_2^{-1} &=& g_2^{-1} g_3 g_5^{-t_{235}}. \\
\end{eqnarray*}
These conjugates allow to evaluate the relations $(R5)$-$(R10)$ and thus
to determine that the presentation of $G(t)$ is consistent.
\end{proof}

\section{Hall polynomials}
\label{hallpoly}

Next, we consider Hall polynomials for the $\T$-groups of Hirsch length at
most $5$. We assume that the considered groups are given by consistent 
presentations of the form $G(t)$ as exhibited in Section \ref{presen}.
We denote $s_2(x) = x(x-1)/2$ and $s_3(x) = x(x-1)(x-2)/6$. 

\begin{theorem} \label{hall}
Let $t \in \Z^{10}$ so that $G(t)$ is consistent. Then the multiplication
in $G(t)$ can be described by Hall polynomials $p_i(a,b)$ with $1 \leq i
\leq 5$ and such polynomials are given by
\begin{eqnarray*}
p_1 &=& a_1 + b_1, \\
p_2 &=& a_2 + b_2, \\
p_3 &=& a_3 + b_3 + t_{123} a_2 b_1,  \\
p_4 &=& a_4 + b_4 + t_{124} a_2 b_1 + t_{134} a_3 b_1 + t_{234} a_3 b_2 \\
    && + t_{123} t_{134} a_2 s_2(b_1) + t_{123} t_{234} s_2(a_2) b_1
       + t_{123} t_{234} a_2 b_1 b_2, \\
p_5 &=& a_5 + b_5
+ t_{345} a_4 b_3
+ t_{245} a_4 b_2
+ t_{235} a_3 b_2 
+ t_{145} a_4 b_1
+ t_{135} a_3 b_1
+ t_{125} a_2 b_1 \\
&&
+ t_{234} t_{345} s_2(a_3) b_2
+ t_{234} t_{245} a_3 s_2(b_2) 
+ t_{134} t_{345} s_2(a_3) b_1
+ t_{134} t_{145} a_3 s_2(b_1) \\
&&
+ t_{234} t_{345} a_3 b_2 b_3
+ t_{134} t_{345} a_3 b_1 b_3 
+ t_{134} t_{245} a_3 b_1 b_2 
+ t_{124} t_{345} a_2 b_1 b_3 \\
&&
+ t_{124} t_{345} a_2 a_3 b_1 
+ (t_{123} t_{235}+t_{124} t_{245}) a_2 b_1 b_2  \\
&&
+ (t_{123} t_{235}+t_{124} t_{245}) s_2(a_2) b_1 
+ (t_{123} t_{135} +t_{124} t_{145}) a_2 s_2(b_1) \\
&&
+ t_{123} t_{234} t_{245} a_2 b_1 s_2(b_2) 
+ t_{123} t_{134} t_{245} s_2(a_2) s_2(b_1) 
+ t_{123} t_{234} t_{245} s_3(a_2) b_1 \\
&&
+ t_{123} t_{134} t_{145} a_2 s_3(b_1) 
+ t_{123} t_{234} t_{245} s_2(a_2) b_1 b_2
+ t_{123} t_{134} t_{245} a_2 s_2(b_1) b_2.
\end{eqnarray*}
\end{theorem}

\begin{proof}
We use the approach of Sims \cite[page 441ff]{Sim94} to determine the 
desired polynomials. As a first step, we determine the polynomials 
$r_{i,j,k}$ with
\[ g_i^x g_j^y = g_j^y g_i^x
         g_{i+1}^{r_{i,j,i+1}(x,y)} \cdots g_n^{r_{i,j,n}(x,y)}\]
for $1 \leq j < i \leq 5$ and all $x, y \in \Z$. Using
\cite[5.1.5]{Rob82}, we note that the following conditions hold
for arbitrary group elements $g,h$ and $x \in \Z$:
\begin{eqnarray*}
(1) \ [[g,h],g] &=& 1 \mbox{ implies that } [g^x,h]=[g,h]^x,  \\
(2) \ [[g,h],h] &=& 1 \mbox{ implies that }  [g,h^x]=[g,h]^x,\\
(3) \ [[g,h],g] &=& [[g,h],h]=1 \mbox{ implies that } 
          (gh)^x=g^xh^x[g,h]^{s_2(x)}.
\end{eqnarray*}
As $[g_4, g_i]$ is central in $G(t)$, we thus obtain that
\begin{eqnarray*}
g_4^x g_1^y &=& g_1^yg_4^x [g_4^x,g_1^y]\\
            &=& g_1^yg_4^x [g_4,g_1^y]^x\\
            &=& g_1^yg_4^x [g_4,g_1]^{xy}\\
            &=& g_1^y g_4^x g_5^{xy t_{145}}, \\
            && \ \\
g_4^x g_2^y &=& g_2^y g_4^x g_5^{xy t_{245}}, \\
g_4^x g_3^y &=& g_3^y g_4^x g_5^{xy t_{345}}. 
\end{eqnarray*}
Next, we compute the equations for $g_3^x g_1^y$ and $g_3^x g_2^y$ in 
two steps. First, we show that induction on $y$ yields that
\[(*)\;\;\;g_3^{g_1^y} = 
  g_3 g_4^{y t_{134}} g_5^{ yt_{135} + s_2(y) t_{134} t_{145}}.\]
This equation is clearly valid for $y = 0$ and $y = 1$. Assume that it 
is valid for $y-1$. Then 
  \begin{eqnarray*}
 g_3^{g_1^y} &=& (g_3^{g_1^{y-1}})^{g_1}\\
 &=& (g_3 g_4^{(y-1)t_{134}} 
      g_5^{(y-1)t_{235}+s_2(y-1)t_{134}t_{145}})^{g_1}\\
 &=&  g_3^{g_1} (g_4^{g_1})^{(y-1)t_{134}} 
      (g_5^{g_1})^{(y-1)t_{135}+s_2(y-1)t_{134}t_{145}}\\
 &=&  g_3g_4^{t_{134}}g_5^{t_{135}} (g_4g_5^{t_{145}})^{(y-1)t_{134}} 
      g_5^{(y-1)t_{134}+s_2(y-1)t_{134}t_{145}}\\
 &=&  g_3g_4^{t_{134}}g_5^{t_{135}} 
      g_4^{(y-1)t_{134}}g_5^{(y-1)t_{134}t_{145}} 
      g_5^{(y-1)t_{135}+s_2(y-1)t_{134}t_{145}}\\
 &=&  g_3 g_4^{yt_{134}} 
      g_5^{yt_{135}+s_2(y-1)t_{134}t_{145}+(y-1)t_{134}t_{145}}\\
 &=&  g_3 g_4^{yt_{134}} 
      g_5^{yt_{135}+((y-2)(y-1)/2)t_{134}t_{145}+(y-1)t_{134}t_{145}}\\
 &=&  g_3 g_4^{yt_{134}} g_5^{yt_{135}+(y(y-1)/2)t_{134}t_{135}}\\
 &=&  g_3 g_4^{yt_{134}} g_5^{yt_{135}+s_2(y)t_{134}t_{135}}.
\end{eqnarray*}
This proves the desired formula for $y \geq 0$. An analogue computation 
leads to $g_3^{g_1^y}=g_3 g_4^{yt_{134}} 
g_5^{yt_{135}+s_2(y)t_{134}t_{145}}$ for $y \leq 0$.\\
Next note that $(g_3^x)^{g_1^y} = (g_3^{g_1^y})^x$ holds. Hence (3) and
$(*)$ yield
\begin{eqnarray*}
(g_3^x)^{g_1^y} 
   &=& (g_3^{g_1^y})^x \\
   &=& (g_3 g_4^{y t_{134}} g_5^{y t_{135} + s_2(y)t_{134} t_{145}})^x \\
   &=& (g_3 g_4^{y t_{134}})^x g_5^{xyt_{135} + x s_2(y)t_{134} t_{145}} \\
   &=& g_3^x g_4^{xy t_{134}} [g_3, g_4]^{s_2(x)y t_{134}} 
       g_5^{xy t_{135} + xs_2(y) t_{134} t_{145}} \\
   &=& g_3^x g_4^{xy t_{134}} 
       g_5^{s_2(x)y t_{134} t_{345} + xyt_{135} + x s_2(y)t_{134} t_{145}}. 
\end{eqnarray*}
Using this and the similar calculation for $g_3^x g_2^y$ we obtain that
\begin{eqnarray*}
g_3^x g_1^y &=& g_1^y g_3^x g_4^{xy t_{134}}
       g_5^{s_2(x)y t_{134} t_{345} + xy t_{135} + x s_2(y) t_{134} t_{145}}, \\
g_3^x g_2^y &=& g_2^y g_3^x g_4^{xy t_{234}}
       g_5^{s_2(x)y t_{234} t_{345} + xy t_{235} + x s_2(y) t_{234} t_{245}}.
\end{eqnarray*}
We have now determined all polynomials $r_{ijk}$ in the subgroup $G_2 =
\langle g_2, \ldots, g_5 \rangle$. Based on this, we can use collection 
to determine the Hall polynomials of $G_2$ with this:
\begin{eqnarray*}
 g^ag^b &=& g_2^{a_2}g_3^{a_3}g_4^{a_4}g_5^{a_5}
            g_2^{b_2}g_3^{b_3}g_4^{b_4}g_5^{b_5}\\
        &=& g_2^{a_2}g_3^{a_3}g_2^{b_2}g_4^{a_4}
            g_5^{a_4b_2t_{245}}g_5^{a_5}g_3^{b_3}g_4^{b_4}g_5^{b_5}\\
        &=& g_2^{a_2+b_2}g_3^{a_3} g_4^{a_3b_2t_{234}+a_4}
            g_5^{s_2(a_3)b_2t_{234}t_{345}+a_3b_2t_{235}
                                    +a_3s_2(b_2)t_{134}t_{245}}
            g_5^{a_4b_2t_{245}+a_5}g_3^{b_3}g_4^{b_4}g_5^{b_5}\\
        &=& g_2^{a_2+b_2}g_3^{a_3+b_3} g_4^{a_3b_2t_{234}+a_4}
            g_5^{b_3a_3b_2t_{234}t_{345}+b_3a_4t_{345}}\\
        &&  g_5^{s_2(a_3)b_2t_{234}t_{345}+a_3b_2t_{235}
              +a_3s_2(b_2)t_{134}t_{245}+a_4b_2t_{245}+a_5}
                    g_4^{b_4}g_5^{b_5}\\
        &=& g_2^{a_2+b_2}g_3^{a_3+b_3} g_4^{a_3b_2t_{234}+a_4+b_4}\\
        &&  g_5^{b_3a_3b_2t_{234}t_{345}+b_3a_4t_{345}
               +s_2(a_3)b_2t_{234}t_{345}
               +a_3b_2t_{235}+a_3s_2(b_2)t_{134}t_{245}
               +a_4b_2t_{245}+a_5+b_5}.
 \end{eqnarray*}
Thus the  the multiplication in $G_2$ is given by
\[ g_2^{a_2} \cdots g_5^{a_5} \cdot g_2^{b_2} \cdots g_5^{b_5}
  = g_2^{f_2(a,b)} \cdots g_5^{f_5(a,b)}\]
with
\begin{eqnarray*}
f_2(a,b) &=& a_2+b_2, \\
f_3(a,b) &=& a_3+b_3, \\
f_4(a,b) &=& a_4+b_4+a_3b_2t_{234}, \\
f_5(a,b) &=& a_5+b_5 +a_3 b_2 t_{235} +a_4 b_2 t_{245} +a_4 b_3 t_{345} \\
         && + a_3 s_2(b_2) t_{234} t_{245}
            + s_2(a_3) b_2 t_{234} t_{345}
            + a_3 b_2 b_3 t_{234} t_{345}.
\end{eqnarray*}
Further, these multiplication polynomials allow to determine powering 
polynomials for $G_2$. The resulting polynomials with
\[ (g_2^{a_2} \cdots g_5^{a_5})^x = g_2^{k_2(a,x)} \cdots g_5^{k_5(a,x)}\]
are given by
\begin{eqnarray*}
k_2(a,x)&=& f_2(k_2(a,x-1),a)\\
        &=& k_2(a,x-1)+a_2\\
        &\vdots& \\
        &=& k_2(a,1)+(x-1)a_2\\
        &=& x a_2, \\ 
        && \ \\
k_3(a,x)&=& f_3(k_3(a,x-1),a)\\
        &=& k_3(a,x-1)+a_3\\
        &\vdots&\\
        &=& k_3(a,1)+(x-1)a_3\\
        &=& x a_3, \\
        && \ \\
k_4(a,x)&=& f_4(k_4(a,x-1),a)\\
        &=& k_4(a,x-1)+a_4+k_3(a,x-1)a_2t_{234}\\
        &=& k_4(a,x-1)+a_4+(x-1)a_3a_2t_{234}\\
        &=& k_4(a,x-2)+2a_4+(x-1+x-2)a_3a_2t_{234}\\
        &\vdots&\\
        &=& k_4(a,1)+(x-1)a_4+((x-1)x-\sum_{i=1}^{x-1}{i})a_3a_2t_{234}\\
        &=& x a_4 + s_2(x) a_2 a_3 t_{234}, \\
        && \ \\
k_5(a,x)&=& f_5(k_5(a,x-1))\\
        &=& k_5(a,x-1)+a_5+k_3(a,x-1)a_2t_{235}+k_4(a,x-1)a_2t_{245}\\
        &&  +k_4(a,x-1)a_3t_{345}+k_3(a,x-1)s_2(a_2)t_{234}t_{245}\\
        &&  +s_2(k_3(a,x-1))a_2t_{234}t_{345}+k_3(a,x-1)a_2a_3t_{234}t_{345}\\
        &=& k_5(a,x-1)+a_5+(x-1)a_3a_2t_{235}
            +((x-1)a_4+s_2(x-1)a_2a_3t_{234})a_2t_{245}\\
        &&  +((x-1)a_4+s_2(x-1)a_2a_3t_{234})a_3t_{345}
            +(x-1)a_3s_2(a_2)t_{234}t_{245}\\
        &&  +(a_3^2s_2(x-1)+(x-1)s_2(a_3))a_2t_{234}t_{345}
            +(x-1)a_3a_2a_3t_{234}t_{345}\\   
        &=& k_5(a,x-2)+2a_5
            +(x-1+x+1)(a_2a_3t_{235}+a_2a_4t_{245}+a_3a_4t_{345})\\
        &&  +(s_2(x-2)+s_2(x-2))(a_2^2a_3t_{234}t_{245}+a_2a_3^2t_{234}t_{345}
                                 +a_3^2a_2t_{234}t_{345})\\
        &&  +(x-1+x-2)(s_2(a_3)a_2t_{234}t_{345}+a_2a_3^2t_{234}t_{345}
                      +s_2(a_2)t_{234}t_{345})\\
        &=& \cdots\\
        &\vdots&\\
        &=& k_5(a,1)+(x-1)a_5
            s_2(x) (a_2a_3t_{235}+a_2a_4t_{245}+a_3a_4t_{345})\\
        &&  +s_3(x)(a_2^2a_3t_{234}t_{245}+a_2a_3^2t_{234}t_{345}
                     +a_3^2a_2t_{234}t_{345})\\
        &&  +s_2(x) (s_2(a_3)a_2t_{234}t_{345}+a_2a_3^2t_{234}t_{345}
                      +s_2(a_2)t_{234}t_{345})\\        
        &=& x a_5 
        + s_2(x) (a_2 a_3 t_{235} + a_2 a_4 t_{245} + a_3 a_4 t_{345} ) \\
         && + s_2(x) a_2 a_3 t_{234} 
          (t_{245} ((2x-1)a_2-3) + t_{345} ((4x+1)a_3 -3))/6.
\end{eqnarray*}
We now consider the computation for $g_2^x g_1^y$. This is done in two steps.
At first we investigate the conjugate $g_2^{g_1^y}$ and write this as
$g_2 g_3^{r(y)} g_4^{s(y)} g_5^{t(y)}$ for polynomials $r(y), s(y), t(y)$. 
We determine these recursively using the polynomials $f_2,\dots, f_5$: 
\begin{eqnarray*}
(g_2)^{g_1^{y}} &=& ((g_2)^{g_1^{y-1}})^{g_1}\\
  &=& (g_2 g_3^{r(y-1)} g_4^{s(y-1)} g_5^{t(y-1)})^{g_1}\\
  &=& g_2^{g_1}(g_3^{g_1})^{r(y-1)}(g_4^{g_1})^{s(y-1)}
      (g_5^{g_1})^{t(y-1)}\\
  &=& g_2g_3^{t_{123}}g_4^{t_{124}}g_5^{t_{125}}
      g_3^{r(y-1)}g_4^{r(y-1)t_{134}}
          g_5^{s_2(r(y-1))t_{134}t_{345}+r(y-1)t_{135}}
      g_4^{s(y-1)}g_5^{s(y-1)t_{145}}\\
  &&  g_5^{t(y-1)}\\
  &=& (g_2g_3^{t_{123}}g_4^{t_{124}}g_5^{t_{125}})
      (g_3^{r(y-1)}g_4^{r(y-1)t_{134}+s(y-1)}\\
  &&    g_5^{s_2(r(y-1))t_{134}t_{345}+r(y-1)t_{135}
             +s(y-1)t_{145}+t(y-1)})\\
  &=& g_2^{1+0} g_3^{t_{123}+r(y-1)}
      g_4^{t_{124}+r(y-1)t_{134}+s(y-1)}\\
  &&  g_5^{t_{125}+t(y-1)+s_2(r(y-1))t_{134}t_{345}         
           +r(y-1)t_{135}+s(y-1)t_{145}
           +t_{124}r(y-1)t_{345}}.    
\end{eqnarray*}
We obtain the following conditions: 
\begin{eqnarray*}
r(y) &=& r(y-1) + t_{123}, \\
s(y) &=& s(y-1) + t_{124} + r(y-1) t_{134}, \\
t(y) &=& t(y-1) + t_{125} 
            + s_2(r(y-1)) t_{134} t_{345} 
            + r(y-1) (t_{124} t_{345} + t_{135})
            + s(y-1) t_{145}. 
\end{eqnarray*}

Solving these recursions and using that $t_{123} t_{345} = 0$ yields
\begin{eqnarray*}
r(y) &=& r(y-1)+t_{123}\\
     &=& r(y-2)+t_{123}+t_{123}\\
     &\vdots& \\
     &=& r(1)+ (y-1) t_{123}\\
     &=& y t_{123}, \\
     && \ \\
s(y) &=& s(y-1)+ t_{124} + (y-1)t_{123}t_{134}\\
     &=& s(y-2) +2t_{124} + (y-1+y-2) t_{123}t_{134}\\
     &\vdots& \\
     &=& s(1)+ (y-1) t_{124}+ ((y-1)y-\sum_{i=1}^{y-1}{i})t_{123}t_{134}\\
     &=& yt_{124} +(y(y-1)/2)t_{123}t_{134}\\
     &=& y t_{124} + s_2(y) t_{123} t_{134}, \\
     && \ \\
t(y) &=& t(y-1)+t_{125}+(t_{123}^2s_2(y-1)+(y-1)s_2(t_{123}))t_{134}t_{345}\\
     &&  +(y-1)t_{123}(t_{124}t_{345}+t_{135})+(y-1)t_{124}t_{145}
         +s_2(y-1)t_{123}t_{134}t_{145}\\
     &=& t(y-2)+2t_{125}
         + (s_2(y-1)+s_2(y-2))t_{123}^2t_{134}t_{345}\\
     &&  +(y-1+y-2)s_2(t_{123})t_{134}t_{345}
         +(y-1+y-2) t_{123}(t_{124}t_{345}+t_{135}) \\
     &&  +(y-1+y-2) t_{124}t_{145}
         +(s_2(y-1)+s_2(y-2))t_{123}t_{134}t_{145}\\
     &=& \cdots\\
     &\vdots&\\    
     &=& t(1)+(y-1)t_{125}\\
     &&  +(y(y-1)/2)(s_2(t_{123})t_{134}t_{345}+t_{123}(t_{124}t_{345}+t_{135})
                     +t_{124}t_{145})\\
     &&  +(\sum_{i=1}^{y-1}{s_2(i)})t_{123}t_{134}(t_{123}t_{345}+t_{145})\\            
     &=& y t_{125} 
      + s_2(y)( s_2(t_{123}) t_{345} t_{134} 
           + t_{123} (t_{124} t_{345} + t_{135})
           + t_{124} t_{145} ) \\
     &&
      + s_3(y) t_{123} t_{134} (t_{123} t_{345} + t_{145} )\\
       &=& y t_{125} 
            + s_2(y)(t_{123} t_{135} + t_{124} t_{145} ) 
            + s_3(y) t_{123} t_{134} t_{145}.
\end{eqnarray*}
Using the polynomials $k_2, \ldots, k_5$ we now determine
\[ g_2^x g_1^y
  = g_1^y (g_2^{g_1^y})^x
  = g_1^y (g_2 g_3^{r(y)} g_4^{s(y)} g_5^{t(y)})^x
  = g_1^y g_2^x g_3^{R(x,y)} g_4^{S(x,y)} g_5^{T(x,y)} \]
with
\begin{eqnarray*}
R(x,y) &=& x r(y) = xy t_{123}, \\
S(x,y) &=& x s(y) + s_2(x) r(y) t_{234} \\
       &=& xy t_{124} + x s_2(y) t_{123} t_{134} + s_2(x) y t_{123} t_{234}, \\
T(x,y) &=&
     x t(y) + s_2(x) (r(y) t_{235} + s(y) t_{245} + r(y) s(y) t_{345} ) \\
    && + s_2(x) r(y) t_{234} (t_{245} ((2x-1)-3) + t_{345} ((4x+1)r(y) -3))/6.
\end{eqnarray*}
Thus we have determined all polynomials $r_{i,j,k}$. Based on these, one
can use the following GAP program to perform a symbolic collection and 
thus obtain $p_1, \ldots, p_5$. This program assumes that $r$ and $t$ are
global variables so that $r[i][j][k]$ contains the polynomial $r_{ijk}$
and $t[i][j][k]$ is the indeterminate $t_{ijk}$.

\begin{verbatim}
CollectSymbolic := function( )
    local a, b, v, stack, c, i, j, k, m;

    # initiate variables
    a := List([1..5],  
         x -> Indeterminate(Rationals, Concatenation("a",String(x))));
    b := List([1..5], 
         x -> Indeterminate(Rationals, Concatenation("b",String(x))));
    v := StructuralCopy(a);
    v[5] := v[5] + b[5];
    stack := Reversed(List([1..4], x -> [x, b[x]]));

    # perform collection until stack is empty
    while Length(stack) > 0 do
        c := stack[Length(stack)];
        Unbind(stack[Length(stack)]);
        if c[2] <> 0*c[2] then
            i := c[1];
            v[i] := v[i] + c[2];
            for j in Reversed([i+1..4]) do
                for k in Reversed([1..Length(r[i][j])]) do
                    m := StructuralCopy(r[i][j][k]);
                    m[2] := Value(m[2], [x,y], [v[j], c[2]]);
                    if m[1] = 5 then
                        v[5] := v[5] + m[2];
                    elif m[2] <> 0*m[2] then
                        Add(stack, m);
                    fi;
                od;
                v[j] := 0;
            od;
        fi;
    od;

    # divide by consistency relations to simplify polynomial
    v[5] := PolynomialReduction(v[5], [t[1][2][3]*t[3][4][5]], 
                                             MonomialLexOrdering())[1];
    return v;
end;
\end{verbatim}
\end{proof}

If $G$ is a $\T$-group of Hirsch length $n < 5$, then Hall polynomials for 
$G$ can also be read off from Theorem \ref{hall}: These are given by 
$p_1, \ldots, p_n$.

\def\cprime{$'$} \def\cprime{$'$}

\end{document}